\documentclass[a4paper,11pt]{amsart}

\usepackage[latin1]{inputenc}
\usepackage[T1]{fontenc}
\usepackage[english]{babel}

\usepackage{amsmath,amssymb,amsthm,enumerate,xspace}
\usepackage{comment}
\usepackage{accents}
\usepackage[colorlinks=true,citecolor=blue,linkcolor=blue]{hyperref}
\usepackage{mathrsfs}
\usepackage{cleveref}

\usepackage{varwidth}

\usepackage{mathrsfs}
\usepackage{amscd}
\usepackage{amsfonts}
\usepackage{amsmath}
\usepackage{amssymb}
\usepackage{amstext}
\usepackage{amsthm}
\usepackage{amsbsy}

\usepackage{xspace}
\usepackage[all]{xy}
\usepackage{graphicx}
\usepackage{url}
\usepackage{latexsym}

\usepackage{cleveref}
\usepackage{tikz-cd}

\newtheorem{theo}{Theorem}[section]
\newtheorem{lem}[theo]{Lemma}
\newtheorem{prop}[theo]{Proposition}
\newtheorem{cor}[theo]{Corollary}

\newtheorem{innercustomgeneric}{\customgenericname}
\providecommand{\customgenericname}{}
\newcommand{\newcustomtheorem}[2]{%
  \newenvironment{#1}[1]
  {%
   \renewcommand\customgenericname{#2}%
   \renewcommand\theinnercustomgeneric{##1}%
   \innercustomgeneric
  }
  {\endinnercustomgeneric}
}

\newcustomtheorem{customthm}{Theorem}
\newcustomtheorem{customprop}{Proposition}
\newcustomtheorem{customlemma}{Lemma}
\newcustomtheorem{customcorollary}{Corollary}

 \theoremstyle{definition}
\newtheorem{dfn}[theo]{Definition}

\newtheorem{ex}[theo]{Example}

\newtheorem{fact}[theo]{Fact}
\newtheorem{rem}[theo]{Remark}
\newcommand{\N}{\ensuremath{\mathbb{N}}}  
\newcommand{\Z}{\ensuremath{\mathbb{Z}}}

\newcommand{\R}{\ensuremath{\mathbb{R}}}

\newcommand{\M}{\ensuremath{\mathcal{M}}}
\newcommand{\Rs}{\ensuremath{\mathcal{R}}}
\newcommand{\Ks}{\ensuremath{\mathcal{K}}}

\newcommand{\vs}{\vspace{0.2cm}}

\newcommand{\mtf}{\ensuremath{\mathcal{N}}}

\newcommand{\df}{\rangle_\mathrm{def}}

\newcommand{\inv}{^{-1}}

\newcommand{\la}{\langle}
\newcommand{\ra}{\rangle}

\DeclareMathOperator{\GL}{GL} 
\DeclareMathOperator{\SL}{SL}

\DeclareMathOperator{\UT}{UT}

\renewcommand{\geq}{\geqslant}

\begin{document}

\author{Annalisa Conversano}

\title{A Jordan-Chevalley decomposition beyond algebraic groups}

\address{School of Mathematical and Computational Sciences,
Massey University, Auckland, NZ}

\email{a.conversano@massey.ac.nz }

\noindent
\date{3 April 2025  \\
 \emph{2020MSC}: 03C64, 20G07, 22E15. \emph{Keywords}: Definable groups; o-minimality; Levi decomposition; Jordan-Chevalley decomposition; $p$-Sylow subgroups}

\maketitle

\begin{abstract} 
We prove a decomposition of definable groups in o-minimal structures generalizing   
the Jordan-Chevalley decomposition of linear algebraic groups. It follows that any definable linear group $G$ is a semidirect product of its maximal normal definable torsion-free subgroup $\mtf(G)$ and a definable subgroup $P$, unique up to conjugacy, definably isomorphic to a semialgebraic group.    

Along the way, we establish two other fundamental decompositions of classical groups in arbitrary o-minimal structures: 1) a Levi decomposition and 2) a key decomposition of disconnected groups, relying on a generalization of Frattini's argument to the o-minimal setting.  In o-minimal structures, together with $p$-groups, $0$-groups play a crucial role. We give a characterization of both classes and show that definable $p$-groups are solvable, like finite $p$-groups, but they are not necessarily nilpotent. Furthermore, we prove that definable $p$-groups ($p=0$ or $p$ prime) are definably generated by torsion elements and, in definably connected groups, $0$-Sylow subgroups coincide with $p$-Sylow subgroups for each $p$ prime.

 \end{abstract}

\thispagestyle{empty}

\vs
\section{Introduction}

\noindent
The Jordan decomposition allows us to write any invertible matrix over a perfect field as a product of a unipotent matrix by a diagonalizable matrix. In a linear algebraic group, the two factors of the decomposition belong to the group as well and do not depend on the linear representation, yielding the so-called Jordan-Chevalley decomposition: every linear algebraic group is a semidirect product of its unipotent radical and a reductive group (see, for example, Chap 1, Section 4 in \cite{borel}). We recall below a few more details about the subgroups component of the decomposition. 

\begin{fact} \label{fact:real-algebraic}
Let $G$ be a linear algebraic group. Then $G$ can be decomposed as
\[
G = N \rtimes P
\]

\noindent
where $N$ is a closed connected nilpotent torsion-free group and $P = F H$, where $F < P$ is finite and $H = P^0$ is the product of a semisimple group $S = [H, H]$ by an algebraic torus $T = Z(H)^0$. Any other semidirect cofactor of $N$ is a conjugate of $P$.
 \end{fact}

This paper discusses the above decomposition for groups definable  in an arbitrary o-minimal structure \M\ (see \Cref{theo:decomposition},   \Cref{theo:Nsplitting} and \Cref{cor:cor-linear}).
This class includes, but it is not limited to, finite groups, Nash groups, algebraic groups over an algebraically closed field of characteristic $0$, semialgebraic groups over a real closed field, real or complex elliptic curves, several classes of connected Lie groups such as abelian, simply connected triangular, compact, semisimple with finite center and various extensions of the above. It is to be noted that many of these Lie groups, including the one presented in \Cref{ex:NcapS-infinite}, cannot be equipped with an algebraic group structure (nor they admit a faithful finite-dimensional representation), and the results of this paper go well beyond the algebraic (or linear) setting. However, this paper adds to a considerable volume of work  by several authors showing that definable groups in o-minimal structures are very closely related to algebraic groups.  See, for instance, \cite{BBO19, Berarducci-Mamino, BMO10, me-nilpotent,  HPP1, m2,  PPSI, PPSII, PPSIII}.

\medskip
When $G$ is connected, the Jordan-Chevalley decomposition is a refinement of the \emph{Levi decomposition} $G = RS$ into solvable by semisimple, where \emph{the solvable radical} $R$ is the largest solvable normal connected subgroup of $G$ (corresponding to $N \rtimes T$ in \Cref{fact:real-algebraic}) and $S$ is a \emph{Levi subgroup}, that is, a maximal semisimple subgroup, unique up to conjugation. Recall that a (definable) group $S$ is called \emph{semisimple} when the trivial subgroup is the only abelian (definably) connected normal (definable) subgroup. Equivalently, $S$ contains no infinite solvable normal (definably) connected (definable) subgroup (see, for instance Fact 3.1 in \cite{COP}). Semisimple (definable) groups are assumed to be (definably) connected.

In o-minimal structures, every definably connected group $G$ has a largest solvable normal \emph{definably connected} subgroup $R$ (that by analogy we call its solvable radical) and the quotient $G/R$ is a definable semisimple group (\cite[Remark 4.1]{diagram}). But there is not always a \emph{definable semisimple} group $S$ such that $G = RS$. The following is a small modification of Example 2.10 in \cite{CPI}.

\begin{ex} \label{ex:NcapS-infinite}
Let $\pi \colon \widetilde{\SL}_2(\R) \to \SL_2(\R)$ be the universal covering map of $\SL_2(\R)$ and let  $s \colon \SL_2(\R) \to \widetilde{\SL}_2(\R)$ be a section of $\pi$. 
Recall that $\widetilde{\SL}_2(\R)$ is a connected semisimple Lie group with infinite cyclic center and $\pi$ is a group homomorphism with central kernel. Thus the image of the 2-cocycle $h_s \colon \SL_2(\R)^2 \to \widetilde{\SL}_2(\R)$, given by $h_s(x, y) = s(x)s(y)s(xy)^{-1}$, is in the center of  $\widetilde{\SL}_2(\R)$. 

 Set $\mathcal{M} = (\R, <, +, \cdot)$. By \cite[\S 8]{HPP}, there is a \M-definable set $X \subset  \widetilde{\SL}_2(\R)$ and a \M-definable map $s \colon \SL_2(\R) \to X$ that is a section of $\pi$,
and $h_s$ is a \M-definable map whenever $s$ is \M-definable.

Let $G = \R\ \times\ \SL_2(\R)$ and consider the group operation on $G$ given by
\[
(a, x) \ast (b, y) = (a + b + h_s(x, y),\ xy).
\]

\vs \noindent  
$(G, \ast)$ is a semialgebraic Lie group with a unique maximal semisimple subgroup $S = [G, G] = \Z  \times \SL_2(\R)$, isomorphic to $\widetilde{\SL}_2(\R)$  by construction. The subgroup $S$ has an infinite cyclic center, so it is not \M-definable (nor it can be defined in any other o-minimal structure), but it is a countable directed union of \M-definable sets, hence \emph{ind-definable}.
 \end{ex}

It has been shown in 
\cite[Theorem 1.1]{CP-Levi} that, when the structure expands a real closed field, the example above is typical in that $G$ contains maximal \emph{ind-definable semisimple} subgroups $S$, all conjugate to each other, such that $G = RS$. These ind-definable subgroups correspond to the maximal semisimple Lie subalgebras of the Lie algebra of $G$, as it happens for Lie groups.

If the structure does not expand a field, definable groups may not have an associated Lie algebra, and it is not clear how a Levi decomposition can be obtained. In Section 3 we establish a generalization to arbitrary o-minimal structures that describes it 
in a novel way, in terms of minimality of perfect complements of the solvable radical:

\begin{theo} \label{theo:levi}
Every definably connected group $G$ contains a smallest perfect subgroup $S$ such that the restriction of the canonical homomorphism $G \to G/R$ to $S$ is surjective, where $R$ is the solvable radical of $G$.
The subgroup $S$ can be chosen to be ind-definable and locally definably connected. We call any such $S$ an \emph{ind-definable Levi subgroup} of $G$ and 
\[
G = RS
\] 

\vspace{.2cm}
an \emph{ind-definable Levi decomposition}. The ind-definable Levi subgroups of $G$ are precisely the conjugates of $S$.

If the structure expands a field, the ind-definable Levi subgroups coincide with the ind-definable semisimple Levi subgroups.

If the structure expands a group, $S$ is the subgroup generated by the image of a definable section $G/R \to G$ of the canonical homomorphism $G \to G/R$.
\end{theo}

Every definable group $G$ has a \emph{maximal normal definable torsion-free subgroup} $\mtf(G)$ that, like the subgroup $N$ in \Cref{fact:real-algebraic}, is always (definably) connected, (definably) contractible and completely solvable. In particular, both $N$ and $\mtf(G)$ are included in the solvable radical and the quotient is a (maximal) torus in both cases. Unlike the subgroup $N$ in \Cref{fact:real-algebraic} though, 
$\mtf(G)$ does not always have a semidirect cofactor (definable or not) in $G$. For instance, this is the case for the group $G$ in \Cref{ex:NcapS-infinite}, where $\mtf(G) = \R \times \{e\}$.   If $H$ were a cofactor of $\mtf(G)$, then $G = \mtf(G) \times H$, because $\mtf(G)$ is central, and $H \cong \SL_2(\R)$ would be a perfect subgroup coinciding with the commutator subgroup $[G, G] \cong \widetilde{\SL}_2(\R)$, contradiction.

However, even when $\mtf(G)$ is not a semidirect factor of $G$, in Section 5 we find a refinement of the Levi decomposition generalizing \Cref{fact:real-algebraic}, where the role of the unipotent radical $N$ is played by $\mtf(G)$. More precisely:

\begin{theo} \label{theo:decomposition}
Let $G$ be a definable group and set $N = \mtf(G) \subseteq G^0$.

\begin{enumerate}

\item
For each ind-definable Levi subgroup $S$ of $G^0$ there is a $0$-Sylow $T$ of the solvable radical that centralizes $S$ and such that  
\[
G^0 = NTS,
\]

where $R=NT$ is the solvable radical.

\noindent
Conversely, for every $0$-Sylow $T$ of the solvable radical of $G$, there is an ind-definable Levi subgroup $S$ that centralizes $T$, and $G^0 = NTS$. 

\medskip
\item For any decomposition $G^0 = NTS$ as in $(1)$, there is a finite subgroup 
\[
F \subset N_G(T) \cap N_G(S)  
\]

such that $G = FG^0$ and any product of subgroups from $\{F, N, T, S\}$ is still a subgroup of $G$.
\end{enumerate}
\end{theo}

For reader's convenience, \Cref{theo:levi} and \Cref{theo:decomposition} are presented separately, although the proof of the latter relies heavily, among other things, on the proof of the former. One may wonder if a simpler proof of \Cref{theo:decomposition} could be provided when the structure expands a field (and therefore a Levi decomposition is already known regardless of \Cref{theo:levi}). Although possible, we believe it to be unlikely. In fact, the most laborious and fundamental case of the proof appears to be when the ind-definable Levi subgroups are definable and semisimple, where assuming a field structure does not seem to provide an advantage.   

\bigskip
As noted before, the unique ind-definable semisimple Levi subgroup of the group $G$ in \Cref{ex:NcapS-infinite} is not definable. In Section 5 we show that this is a necessary condition for $\mtf(G)$ \emph{not} being a semidirect factor of $G$ or $G^0$:

 \begin{theo} \label{theo:Nsplitting}
 Let $G$ be a definable group such that $G^0$ has a definable semisimple Levi subgroup $S$.

\begin{enumerate}[(i)]

\item The definable exact sequence
\[
1\ \longrightarrow\ \mtf(G)\ \longrightarrow\ G^0\ \longrightarrow\  G^0/\mtf(G)\ \longrightarrow\ 1
\]

splits, and there is a semidirect cofactor $H$ of $\mtf(G)$ in $G^0$ such that $S = [H, H]$ and $H = TS$ for some abstract torus $T \subset Z(H)$ with finite intersection with $S$. 

Moreover, $H$ is definable if and only if all $0$-subgroups of $G$ are definably compact. If this is the case, the definable semidirect cofactors of $\mtf(G)$ are conjugate to each other.

\medskip
\item The
definable exact sequence
\[
1\ \longrightarrow\ \mtf(G)\ \longrightarrow\ G\ \longrightarrow\  G/\mtf(G)\ \longrightarrow\ 1
\]

splits, and it splits definably if and only if the sequence in $(i)$ does.
 
If this is the case, the definable semidirect cofactors of $\mtf(G)$ are conjugate to each other.
\end{enumerate}
\end{theo}

\begin{dfn} \label{linear-def}
We say that a group $G$ definable in an o-minimal structure $\M$ is \emph{linear} when there is a definable real closed field \Rs\ and a definable isomorphism between $G$ and a definable subgroup of $\GL_n(\Rs)$, for some $n \in \N$. 
\end{dfn}

From \Cref{theo:Nsplitting} we can deduce that linear groups admit a \emph{definable} decomposition in striking resemblance to the linear algebraic case:

\begin{cor} \label{cor:cor-linear}
Let $G$ be a definable linear group. Then $G$ can be decomposed as
\[
G = \mtf(G) \rtimes P,
\]

\noindent
where $P$ is a subgroup such that $P = F H$,  where $H = P^0$ is the product of a definable semisimple group $S = [H, H]$ by a definable torus $T = Z(H)^0$, and $F$ is a finite subgroup contained in $N_G(T) \cap N_G(S)$. Any other definable semidirect cofactor of $\mtf(G)$ in $G$ is a conjugate of $P$. 
\end{cor}

\medskip
Like the subgroup $N$ in \Cref{fact:real-algebraic}, the torsion-free subgroup $\mtf(G)$ is (definably) contractible and (definably) completely solvable, but it is not always nilpotent.
A natural question is whether it is possible to refine the decomposition in \Cref{cor:cor-linear} (and the one in  \Cref{theo:decomposition}), 
to find a nilpotent semidirect factor as in \Cref{fact:real-algebraic}. In Section 5 we rule out several natural nilpotent candidates (including the unipotent radical for linear groups), suggesting that \Cref{cor:cor-linear} and \Cref{theo:decomposition} are the best analogue to the Jordan-Chevalley decomposition in the o-minimal setting, for the linear and non-linear case, respectively.

  \medskip
When $G$ is not connected, the finite subgroup $F$ in \Cref{fact:real-algebraic} meets all connected components of $G$. Similarly, in Section 2 we prove that in any definable group $G$ there is always a finite subgroup intersecting every definably connected component. Finding this subgroup with very specific properties will be later essential in the proofs of \Cref{theo:decomposition} and \Cref{theo:Nsplitting}. We prove the following:
 
\begin{theo}\label{theo:disconnected}
Let $G$ be a definable group.  
\begin{itemize}
\item If $E(G) \neq 0$, there is a finite subgroup $F$, unique up to conjugation,   such that $G = G^0 \rtimes F$. 

\vs
\item If $E(G) = 0$, for every $0$-Sylow subgroup $A$ of $G$ there is a finite subgroup $F \subset N_G(A)$ such that $G = F G^0$.  
\end{itemize}
  \end{theo}

The second part of \Cref{theo:disconnected} is reminescent of the Frattini's Argument, a fundamental tool in the study of finite groups. Indeed, the result holds in the o-minimal setting beyond finite groups and it is another important ingredient in the proof of \Cref{theo:decomposition}:

\begin{lem}[Frattini's Argument] \label{lem:frattini}
 Let $G$ be a definable group, $H$ a normal definable subgroup and $S$ a $p$-Sylow subgroup of $H$ ($p = 0$ or $p$ prime). Then $G = N_G(S) H$.
\end{lem}

As witnessed by \Cref{theo:disconnected} and \Cref{lem:frattini}, it is now apparent that  $p$-Sylow subgroups play a fundamental role in the structure of a definable group, similarly to classical $p$-Sylow subgroups in finite groups. In Section 4 the following characterization is established, later crucially used in the proof of \Cref{theo:decomposition}.

\begin{theo}\label{theo:sylow}
Let $H < G$ be definable groups and $p$ a prime number.

\begin{enumerate}

\item The following are equivalent:

\begin{enumerate}[(i)]
\item $H$ is a $p$-Sylow of $G$;  

\item $H^0$ is $0$-Sylow of $G$, $H^0 = H \cap G^0$ and $H/H^0$ is a  $p$-Sylow of $G/G^0$. 

 \end{enumerate}

\item The following are equivalent: 
\begin{enumerate}[(i)]

\item $H$ is a $0$-Sylow of $G$;  

\item $H$ is a maximal abelian definably connected subgroup definably generated by its torsion;

\item $H$ is a $p$-Sylow of $G^0$.
\end{enumerate}

\end{enumerate}
\end{theo}

 \begin{cor}
Let $p$ be a prime number and $G$ a definably connected group. A definable subgroup $H$ of $G$ is a $p$-Sylow of $G$ if and only if $H$ is a $0$-Sylow of $G$.
\end{cor}

Since $0$-groups are abelian and finite $p$-groups are nilpotent, the following holds:

\begin{cor}\label{cor:p-solvable}
Let $p$ be a prime number or $p = 0$. Definable $p$-groups are solvable and definably generated by torsion elements.
\end{cor}

Unlike finite $p$-groups, definable $p$-groups are not necessarily nilpotent, though. In \Cref{ex:inSL2C}, we provide a semialgebraic $2$-group that is not nilpotent.


\medskip
The paper is organized as follows: in Section 2 we recall Strzebonski's work on Sylow's theorems for definable groups, prove \Cref{lem:frattini} and the fundamental  decomposition of disconnected groups given
in \Cref{theo:disconnected}.

Section 3 is dedicated to the proof of the ind-definable Levi decomposition of \Cref{theo:levi}.    

In Section 4 we study $p$-groups. We show that $0$-groups are the definably connected abelian groups generated by their torsion subgroup and, for $p$ prime, definable $p$-groups are extensions of a finite $p$-group by a $0$-group (\Cref{theo:p-groups:list}). A characterization of $p$-Sylow subgroups of any definable group is then provided in the proof of \Cref{theo:sylow}.

In Section 5 the decompositions  of  \Cref{theo:decomposition} are established. Furthermore, we prove in \Cref{theo:Nsplitting} that whenever $G$ has definable semisimple Levi subgroups, $\mtf(G)$ is a semidirect factor in both $G^0$ and $G$.  The definable decomposition for linear groups in \Cref{cor:cor-linear} follows. We provide examples suggesting these decompositions cannot be improved.

\medskip
Unless otherwise stated, throughout the paper groups are definable - \emph{with parameters} - in an arbitrary o-minimal structure \M. Recall that by Edmundo's Theorem 7.2 in \cite{Edmundo}, for every definable subgroup $H$ of a definable group $G$, the set $G/H$ is definable.  

When we say that a group is \emph{definably connected} or \emph{definably compact}, we assume it is definable.

\section{Disconnected groups}

In any definable group $G$, the definably connected component of the identity   $G^0$ is the smallest definable subgroup of finite index \cite[Prop 2.12]{Pillay - groups}. Therefore any definable group is an extension of a finite group by a definably connected group. In this section we show that when $G^0$ is torsion-free the extension splits and, in general, there is a finite subgroup $F$ that meets every definably connected component, so that $G = FG^0$.  
 
Here and almost everywhere else in the paper, the o-minimal Euler's characteristic 
will be an important tool. If $\mathcal{P}$ is a cell decomposition of a definable set $X$, \emph{the o-minimal Euler characteristic} $E(X)$ is defined as the number of even-dimensional cells in 
$\mathcal{P}$ minus the number of odd-dimensional cells in $\mathcal{P}$, it does not depend on $\mathcal{P}$ and it is invariant by definable bijections  (see \cite{Lou}, Chapter 4). As points are $0$-dimensional cells, for finite sets cardinality and Euler characteristic coincide, leading Strzebonski  to prove in \cite{Strzebonski} o-minimal analogues of Sylow's theorem for finite groups.
In definable groups, in addition to $p$-Sylow subgroups, there are $0$-Sylow subgroups, that in \cite{me2} have been proven to play the same role as  maximal tori in Lie groups. 
We recall below some definitions and results we will be using.

\begin{fact}\cite[Lem 2.12]{Strzebonski} \label{fact:products} 
Let $K < H < G$ be definable groups. Then
\begin{enumerate}
\item[(a)] $E(G) = E(H)E(G/H)$

\item[(b)] $E(G/K) = E(G/H)E(H/K)$
\end{enumerate}
\end{fact}

\begin{dfn}\cite{Strzebonski} Let $G$ be a definable group. We say that $G$ is a \emph{$p$-group} if:
\begin{itemize}
\item $p$ is a prime number and for any proper definable $H < G$,
\[
E(G/H) \equiv 0 \quad \mod p,
\]

or 

\medskip
\item $p = 0$ and for any proper definable subgroup $H < G$,
\[
E(G/H) = 0.
\]

\end{itemize}

A maximal $p$-subgroup of a definable group $G$ is called a \emph{$p$-Sylow} of $G$.
\end{dfn}

\begin{fact}\label{fact:str} \cite{Strzebonski} Let $G$ be a definable group and $p$ a prime number or $p = 0$.
\begin{enumerate}
\item If $n$ is a prime dividing $E(G)$, then $G$ contains an element of order $n$. In particular, if $E(G) = 0$ then $G$ has elements of each prime order. Moreover, 
\[
 G \mbox{ is torsion-free } \ \Longleftrightarrow\ |E(G)| = 1.
\]

\item If $E(G) = 0$, then $G$ contains an infinite $0$-subgroup.  

\item Every $0$-group is abelian and definably connected. 

\item Each $p$-subgroup of $G$ is contained in a $p$-Sylow, and $p$-Sylows are all conjugate.   

\item If $H$ is a $p$-subgroup of $G$, then 
\[
H \mbox{ is a $p$-Sylow } \Longleftrightarrow\ E(G/H) \neq 0 \mod p
\]

\item If $A$ is a $0$-Sylow of $G$, then $E(G/N_G(A)) = 1$.

\end{enumerate}
\end{fact}

\noindent
As a natural application of Strzebonski's work, we now show how Frattini's Argument applies to definable groups as well.

\begin{proof} [\textbf{Proof of \Cref{lem:frattini}}]
For $g \in G$, let $S^g$ denote the conjugate of $S$ by $g$. Then $S^g$ is a $p$-Sylow of $H$ for any $g \in G$, because the conjugation map $H \to H$, $x \mapsto gxg^{-1}$ is a definable automorphism of $H$, and the Euler characteristic is invariant by definable bijections. Since $p$-Sylow subgroups are all conjugate, there is $x \in H$ such that $S^g = S^x$ and $gx^{-1} \in N_G(S)$, the normalizer of $S$ in $G$. Therefore $G = N_G(S)H$.
\end{proof}

As observed in the introduction of \cite{PPSI}:

\begin{fact} \label{fact:tricotomy}
If $G$ is definably connected, then either $E(G) = \pm 1$ (iff $G$ is torsion-free) or $E(G) = 0$.
\end{fact}

\begin{fact}  \cite[Cor 2.4, Claim 2.11]{PeSta} \cite[Prop 4.1]{Strzebonski} \label{fact:divisibility}
 \begin{enumerate} 
 
 \item Abelian definably connected groups are divisible.
 
 \item Torsion-free definable groups are solvable, definably connected and uniquely divisible.
 \end{enumerate}
 \end{fact}

\begin{fact}\label{fact:me-solvable} \cite[Prop 2.1]{CPI} \cite[Prop 3.1]{me2} 
\begin{enumerate}

\item Every definable group $G$ has a maximal normal definable torsion-free subgroup $\mtf(G)$.

\item If $G$ is solvable and definably connected, then

\begin{itemize}
\item $G/\mathcal{N}(G)$ is a definable torus (that is, abelian, definably connected and definably compact).

\item For any $0$-Sylow $A$ of $G$, $G = \mathcal{N}(G)A = \mathcal{N}(G) \rtimes T$, where $T$ is any direct complement of $\mathcal{N}(A)$ in $A$. %
\end{itemize}

\noindent
In particular, if $G$ is abelian, then $A$ is unique and $G = \mathcal{N}(G) \times T$.  
\end{enumerate}
\end{fact}

\begin{rem}\label{rem:compact-G0abelian}
Let $G$ be a definably compact group such that $G^0$ is abelian. Then there is a finite subgroup $F$ of $G$ such that $G = FG^0$. 
 \end{rem}

\begin{proof}
Same proof as for compact Lie groups applies. See, for instance the proof of Theorem 6.10(i) in \cite{HM}. 
\end{proof}

\begin{prop} \label{prop:G0torsionfree}
 Let $G$ be a definable group such that $G^0$ is torsion-free. Then $G = G^0 \rtimes F$, for some finite subgroup $F$, unique up to conjugation.
 \end{prop}

\begin{proof}
We will prove our claim by induction on $d = \dim G^0$.

Suppose first $G^0$ is abelian and set $\overline{G} = G/G^0$. Then $G^0$ is a $\overline{G}$-module. As $G^0$ is uniquely divisible (\Cref{fact:divisibility}), then $H^n(\overline{G}, G^0) = 0$ for all $n>0$ (see, for instance, \cite[11.3.8]{robinson}). Taking $n=2$, we can find some finite subgroup $F$ that is a complement of $G^0$. Taking $n=1$, we can deduce that $F$ is unique up to conjugation.

If $d = 1$, $G^0$ is abelian. So assume $d > 1$ and $G^0$ is not abelian.  Let $N$ be the commutator subgroup of $G^0$. As $G^0$ is solvable (\Cref{fact:divisibility}), $N$ is definable by \cite[Theo 1.3]{BJO} and 
$\dim N < \dim G^0$, because $G^0$ is definably connected. By \Cref{fact:str}(1) and \Cref{fact:products}(a), any definable subgroup and definable quotient of a definable torsion-free group is torsion-free as well, so $(G/N)^0 = G^0/N$ is torsion-free   and $G/N = G^0/N \rtimes F_1$ by the abelian case. Let $G_1$ be the pre-image in $G$ of $F_1$. By induction hypothesis, $G_1 = N \rtimes F$, for some finite subgroup $F$, unique up to conjugation. Since $G_1$ maps surjectively onto $G/N$, $G = NG_1$ and $G = G^0 \rtimes F$.

Suppose $H$ is another semidirect cofactor of $G^0$ in $G$ and set $H_1 = N \rtimes H$. We want to show that $H$ is a conjugate of $F$. Let $\pi \colon G \to G/N$ be the natural homomorphism and set $\bar H = \pi(H)$. By the abelian case, $\bar H$ is a conjugate of $F_1$, say $\bar H = F_1^x$. Let $g \in \pi\inv(x)$. Then $F^g \subset H_1 = \pi\inv(\bar H)$. By induction hypothesis, $F^g$ is a conjugate of $H$ in $H_1$ and so $H$ is a conjugate of $F$ in $G$, as claimed.
\end{proof}

\begin{prop} \label{prop:Anormal}
Let $A$ be a $0$-Sylow of a definable group $G$ and denote by $N_G(A)$ its normalizer in $G$. Then $N_G(A)^0$ is solvable, $(N_G(A)/A)^0$ is torsion-free, and
 \[
N_G(A)/A = (\mtf(N_G(A))/\mtf(A)) \rtimes N_G(A)/N_G(A)^0. 
\] 
\end{prop}

\begin{proof}
For ease of notation, set $H = N_G(A)$. Let $R$ be the solvable radical of $H$.
As $A$ is abelian and definably connected, then $A \subseteq R$. If $H^0$ is not solvable, then $H^0/R$ is an infinite semisimple definable group and $E(H^0/R) = 0$ (\Cref{fact:tricotomy}). However, 
\[
E(H^0/A) = E(H^0/R)E(R/A) \neq 0
\]
by \Cref{fact:str}(5), contradiction. It follows that $H^0$ is solvable.

Note that $\mtf(H) = \mtf(H^0)$, since torsion-free definable groups are definably connected (\Cref{fact:divisibility}). Moreover, we claim that $\mtf(H) \cap A = \mtf(A)$.

Since $\mtf(H)$ is torsion-free and $A$ is abelian, clearly $\mtf(H) \cap A \subseteq \mtf(A)$. If $\mtf(H) \cap A \neq \mtf(A)$, then the quotient $\mtf(A)/(\mtf(H) \cap A)$ is an infinite torsion-free definable subgroup of $A/(\mtf(H) \cap A)$. However, because $H^0$ is solvable, $H^0 = \mtf(H) A$ by \Cref{fact:me-solvable}, so $A/(\mtf(H) \cap A)$ is definably isomorphic to the definably compact $H^0/\mtf(H)$, contradiction.

Hence $(H/A)^0 = H^0/A = \mtf(H)/\mtf(A)$. Moreover, the group $H/A$ is 
a definable extension of the finite group 
$(H/A)/(H^0/A) = H/H^0$ by a torsion-free definable group. By   \Cref{prop:G0torsionfree}, the extension splits.
\end{proof}

\begin{prop}\label{prop:disconnected-E0}
Let $G$ be a definable group. For every $0$-Sylow $A$ of $G$ there is a finite subgroup $F \subset N_G(A)$ such that $G = F G^0$.
 \end{prop}

\begin{proof}
If $E(G) \neq 0$, then $A = \{e\}$ is the unique $0$-Sylow of $G = N_G(A)$, $G^0$ is torsion-free and \Cref{prop:G0torsionfree} applies.

Let $E(G) = 0$.
 Suppose first $A$ is normal in $G$. That is, $G = N_G(A)$. By \Cref{prop:Anormal},
\[
G/A = (G/A)^0 \rtimes F'
\] 

\vs \noindent
for some finite $F' < G/A$, unique up to conjugation. Let $K'$ be the pre-image of $F'$ in $G$. Note that $\mtf(A)$ is a definably characteristic subgroup of $A$, hence normal in $G$. The quotient $K = K'/\mtf(A)$ is definably compact because $K'/A$ and $A/ \mtf(A)$ are definably compact (\Cref{fact:me-solvable}). Set $A_1 = A/\mtf(A)$.
By \Cref{rem:compact-G0abelian}, $K = F_1 A_1$, for some finite $F_1 < K$, because $A_1 = K^0$ is abelian. Let $H$ be the pre-image of $F_1$ in $G$. By \Cref{prop:G0torsionfree}, there is some finite subgroup $F$ such that $H = \mtf(A) \rtimes F$. Therefore, $G = FG^0$ because $G = HA$.

Suppose now $N_G(A) \neq G$. Because $A$ is definably connected, $A \subseteq G^0$. By \Cref{lem:frattini}, $G = N_G(A)G^0$. Let $F \subset N_G(A)$ be a finite subgroup from above such that $N_G(A) = F N_G(A)^0$. Since $N_G(A)^0 \subseteq G^0$, it follows that $G = FG^0$. 
 \end{proof}

\begin{proof} [\textbf{Proof of \Cref{theo:disconnected}}]
 If $E(G) \neq 0$, then $G^0$ is torsion-free by \Cref{fact:tricotomy}. In this case, a semidirect complement of $G^0$, unique up to conjugation, is found in \Cref{prop:G0torsionfree}. 

If $E(G) = 0$, then \Cref{prop:disconnected-E0} provides a finite subgroup $F \subset N_G(A)$ that meets all definably connected components of $G$. 
 \end{proof}

\begin{rem}
When $E(G) = 0$, it is not always possible to find a semidirect cofactor of $G^0$, unlike the case where $G^0$ is torsion-free. An example is given below.
\end{rem} 

\begin{ex} \label{ex:inSL2C} Let $\Rs$ be a real closed field and $\Ks$ its algebraic closure. Consider the definable (in \Rs) subgroup $G$ of $\SL_2(\Ks)$ of matrices of the form
\[
\begin{pmatrix}
z & 0 \\
0 & z\inv
\end{pmatrix}, 
\begin{pmatrix}
0 & -z\inv \\
z & 0
\end{pmatrix}
\]
where $z \in \Ks$, $|z| = 1$. Then $G$ is an extension of $\Z/2\Z$ by a $1$-dimensional definable torus. The unique $0$-Sylow subgroup of $G$ is $G^0$, corresponding to the diagonal matrices above on the left. Each anti-diagonal matrix on the right has order $4$ and the subgroup generated by any of them meets both definably connected components of $G$. $G^0$ is not a semidirect factor of $G$, as all elements of order $2$ are contained in $G^0$.
 \end{ex}


\section{A Levi decomposition}

The semialgebraic group $G$ from \Cref{ex:NcapS-infinite} has a decomposition $G = RS$, where $R = \R \times \{e\}$ is the solvable radical, and $S = \Z \times \SL_2(\R) \cong \widetilde{\SL}_2(\R)$ is a maximal semisimple group. The Lie algebra of $S$ is semisimple, and a maximal semisimple Lie subalgebra of the Lie algebra of $G$.  Moreover, $S$ is the smallest subgroup of $G$ on which the canonical homomorphism $G \to G/R$ is surjective. The subgroup $S$ cannot be defined in any o-minimal structure over $\R$, as its center is an infinite cyclic group. However, $S$ is a countable directed union of semialgebraic sets, it is \emph{ind-definable} and \emph{locally definably connected}. In this section we discuss such decomposition in arbitrary o-minimal structures.

The expressions ind-definable, $\bigvee$-definable, and locally definable are more or
less synonymous, and refer to definability by a possibly infinite disjunction of
first order formulas. In the context of groups in o-minimal structures these notions (and related connectedness) appear for instance in  \cite{BEP19, BEM13, E06, HPP1, Peterzil-Starchenko1}.

We refer to \cite[Section 2]{CP-Levi} for the precise notion of an ind-definable  group. For our purposes, it will be enough to take a countable directed union
of definable sets $G = \bigcup_{n \in \N}X_n$ such that the group operation restricted to each $X_i \times X_j$ is a definable map with image in some $X_k$.

\begin{dfn}\label{dfn:locally-def-connected}
Let $G$ be an ind-definable group. We say that $G$ is \emph{locally definably connected} if $G$ has no proper subgroup $H$ with the properties that for each definable subset $Z$ of $G$, $Z \cap H$ is definable and $Z$ meets only finitely many distinct cosets of $H$ in $G$.
\end{dfn}

\begin{dfn}\cite[Def 2.3 \& 2.5]{CP-Levi}  \label{dfn:semisimple}
\begin{enumerate}[(i)]

\item Let $X$ be an ind-definable set and $Y$ a subset of $X$. We say that
$Y$ is \emph{discrete} if for any definable subset $Z$ of $X$, $Z \cap Y$ is finite.

\medskip
\item We call $G$ \emph{ind-definable semisimple} if $G$ is ind-definable, definably connected, and a central extension of a definable semisimple group by a discrete group. 

\end{enumerate}
\end{dfn}  

\begin{fact}\cite[Theorem 1.1 \& 1.2]{CP-Levi} \label{fact:CP-levi}
Let $G$ be a definably connected group definable in an o-minimal expansion of a field and let $\mathfrak{g}$ be its Lie algebra.  

\begin{enumerate}

\item $G$ has a maximal ind-definable semisimple subgroup $S$, unique up to conjugation.
Moreover,
\[
G = RS,
\]

\noindent
where $R$ is the solvable radical of $G$, and $Z(S)$ is finitely generated and contains $R \cap S$. We call any such $S$ and \emph{ind-definable semisimple Levi subgroup} of $G$.

\medskip
\item For any semisimple Lie subalgebra $\mathfrak{s}$ of $\mathfrak{g}$ there
is a unique ind-definable semisimple subgroup $S$ of $G$ whose Lie algebra is $\mathfrak{s}$. 
\end{enumerate} 
\end{fact}

Given a group $G$ - definable or not - we denote by $[G, G]_1$ the set of its commutators, by $[G, G]_r$ the set of products of at most $r$ commutators, and by $[G, G] = \langle [G, G]_1 \rangle = \bigcup_{n \in \N} [G, G]_n$ the commutator subgroup of $G$. Recall that a group is called \emph{perfect} when $G = [G, G]$.

\begin{fact}\cite[3.1]{HPP} \label{fact:finite-width}
Let $G$ be a semisimple definable group. Then $G$ is perfect and it has finite commutator width. That is, $G = [G, G]_r$, for some $r \in \N$.
\end{fact}

\begin{fact}\cite[6.1.1]{Lou} \label{fact:section}
Let $\M$ be an o-minimal expansion of an ordered group. Each definable equivalence relation on a definable set $X$ has a definable set of representatives. In particular, any definable surjective map $f \colon X \to Y$ admits a definable section $s \colon Y \to X$.
\end{fact}

\begin{proof}[Proof of \Cref{theo:levi}]
Let $G$ be a definably connected group and $R$ its solvable radical. We will prove our statements by induction on $\dim G$.

Suppose first the structure expands a field and let $\mathfrak{g}$ be the Lie algebra of $G$.  
Fix $S$ an ind-definable semisimple Levi subgroup from \Cref{fact:CP-levi}. Then $S$ is pefect by 
Lemma 2.7 in \cite{CP-Levi}. If $S = \bigcup X_i$, by saturation there is some definable $X_k$ on which 
the canonical projection $X_k \to G/R$ is surjective. Let $s \colon G/R \to X_k$ be a definable section and set $X=s(G/R) \subset S$. We claim that the subgroup generated by $X$ coincides with $S$. 

Set $H = \la X \ra \subseteq S$. Since $S$ is a central extension of $G/R$ and $X$ maps surjectively onto $G/R$, it follows that $H$ is a normal subgroup of $S$ and the quotient $S/H$ is abelian, as it is isomorphic to a quotient of $Z(S)$. However, $S$ is perfect, so $S = H$, as claimed.

To see that $S$ is a minimal (perfect) complement of the solvable radical, suppose $P \subseteq S$ is a (perfect or not) subgroup such that the canonical homomorphism $P \to G/R$ is surjective. As noted before for $H$, $P$ is then a normal subgroup of $S$ with abelian quotient $S/P$ and $S = P$, because $S$ is perfect. 

Conversely, suppose $S$ is an ind-definable Levi subgroup of $G$ and let $\mathfrak{s}$
be its Lie algebra. By the Levi decomposition of  Lie algebras (see, for instance, \S 5.6 in \cite{HN}),  $\mathfrak{s}$
contains a maximal semsimple Lie subalgebra $\mathfrak{s}_1$, that is a maximal semisimple Lie subalgebra of $\mathfrak{g}$ too, because $G = RS$. Let $S_1$ be the   
unique ind-definable semisimple subgroup $S_1$ of $G$ whose Lie algebra is $\mathfrak{s}_1$ from \Cref{fact:CP-levi}. Then $S_1$ is an ind-definable semisimple Levi subgroup of $G$ contained in $S$. As showed above, $S_1$ is a minimal complement of the solvable radical, and it must be $S = S_1$ by minimality of $S$.

Therefore, if the structure expands a field, the ind-definable semisimple Levi subgroups of $G$ in \Cref{fact:CP-levi} are precisely the ind-definable Levi subgroups in our statement.

We now need to show that even when the structure does not expand a field, and $G$ does not necessarily have a well-defined Lie algebra, $G$ contains ind-definable Levi subgroups  and they are all conjugate.

\medskip
Suppose first $G$ is a central extension of a perfect group (definable or not). We can show that in this case the commutator subgroup $S = [G, G]$ of $G$ is the \emph{only} perfect subgroup $P$ of $G$ such that $G = Z P$, where $Z$ is \emph{any} central subgroup of $G$ such that $G/Z$ is perfect. 

Define $\bar G = G/Z$, $N = Z S$ and $\bar N$ the image of $N$ in $\bar G$ under the natural homomorphism $G \to \bar G$. Because 
$Z$ and $S$ are normal subgroups of $G$, so is their product $N$. The quotient $G/N$ is abelian, as it is a quotient of $G/[G, G]$, and we have
\[
G/N = (G/Z)/(N/Z) = \bar{G}/\bar{N}.
\]

However, $\bar G$ is perfect, so it must be $\bar{G} = \bar{N}$
and $G = Z S$. It follows that $[G, G] = [S, S] = S$, and $S$ is perfect.

Suppose $P$ is \emph{any} subgroup such that $G = Z P$. Let $Y \subset P$ be the image of a section $\bar{G} \to G$ of the canonical projection $G \to \bar{G}$. Then for any $a, b \in G$ there are $y_1, y_2 \in Y$ and $z_1, z_2 \in Z$ such that $a = z_1y_1$ and $b = z_2y_2$. Hence $[a, b] = [y_1, y_2] \in P$ and $P$ contains $[G, G]_1$. It follows that $S$ is the \emph{unique} smallest subgroup of $G$ on which the canonical homomorphism $G \to G/Z$ is surjective.

 Suppose now $P$ is any perfect subgroup of $G$. It is easy to see that $P \subseteq S$. Indeed, the group $(PS)/S$ is abelian, as it is a subgroup of $G/[G, G]$. Since
$(PS)/S = P/(P \cap S)$, it follows that $P \cap S = P$, because the quotient of any perfect group is still perfect. Therefore, $S = [G, G]$ is the only perfect subgroup $P$ of $G$ such that $G = Z P$, as claimed.

In particular, when $G$ is a central extension of a definable semisimple group, $R = Z(G)^0$, $G/R$ is perfect (\Cref{fact:finite-width}), and the above applies.

By taking $X_n = [G, G]_n$, we see that $S = \bigcup X_n$ is ind-definable.

Suppose, by way of contradiction, $S$ is not locally definably connected, and let $H$ be a proper subgroup of $S$ violating \Cref{dfn:locally-def-connected}. Let $r$ be the commutator width of $G/R$ from \Cref{fact:finite-width}. The canonical homomorphism $G \to G/R$ is surjective on $X_r$. As $X_r$ meets only finitely many distinct cosets of $H$ in $S$, it follows that the image of $H$ in $G/R$ is an ind-definable subgroup of finite index, and it must be the whole group, because $G/R$ is definably connected.
 However, $S$ is the smallest subgroup of $G$ such that $G = R S$, so $H = S$ and $S$ is locally definably connected. In conclusion, $S$ is the unique ind-definable Levi subgroup of $G$.

If the structure expands a group, let $s \colon G/R \to X_r$ be a definable section of the canonical projection $X_r \to G/R$. The subgroup $P$ generated by $s(G/R)$ in $G$ clearly satisfies $G = RP$, and $P$ contains $S$ because $R \subseteq Z(G)$ and $G/R$ is perfect. Therefore, $P = S = \la s(G/R) \ra$.

\medskip
Assume now $Z(G)$ is finite. We show below that in this case the ind-definable Levi subgroups are definable and semisimple.

The quotient $G/Z(G)$ is centerless, because $G$ is definably connected. By \cite[Theorem 3.1 \& 3.2]{PPSI}, there are definable real closed fields $\Rs_1, \dots, \Rs_k$ such that $G/Z(G)$ is a direct product $G_1 \times \dots \times G_k$ where each $G_i$ is definably isomorphic to a semialgebraic subgroup of $\GL_n(\Rs_i)$, for some $n \in \N$. By \cite[Theorem 4.5]{PPSIII},
each $G_i$ can be decomposed as $G_i = R_iS_i$, where $R_i$ is the solvable radical of $G_i$ and $S_i$ is definable semisimple. Moreover, each $S_i$ is an ind-definable semisimple Levi subgroup in the sense of \Cref{fact:CP-levi}, and unique up to conjugation, by \cite[Lem 4.1 \& 4.2]{CP-Levi}.

Let $S$ be the definably connected component of the pre-image in $G$ of $S_1 \times \dots \times S_k$. Then $S$ is definable and semisimple, because $Z(G)$ is finite.
Note that $R$ is the definably connected component of the pre-image of $R_1 \times \dots \times R_k$, and $G = RS$. By \Cref{fact:finite-width}, $S$ is perfect.

Suppose $P \subseteq S$ is any subgroup such that $G = RP$. Note that because $S$ is definable semisimple, the normal solvable subgroup $R \cap S$ is finite, hence central in $S$. Thus $S$ is a central extension of the definable semisimple group $G/R$ by the finite group $R \cap S$. Since $P$ is contained in $S$ and $G = RP$, $P$ is a central extension of the definable semisimple group $G/R$ by the finite group $R \cap P \subseteq Z(S)$. Hence $P$ is definable,
and $P = S$ because $S$ is definably connected.

Suppose $P$ is another ind-definable Levi subgroup of $G$. We need to show that $P$ is a conjugate of $S$. Define $H = Z(G)P$. As $P$ is perfect, so is $H/Z(G) = P/(P \cap Z(G))$, $P = [H, H]$ and, as proved above, it is the only perfect subgroup of $H$ that maps surjectively onto $H/Z(G)$.

We claim that for each $i = 1, \dots, k$ the image $P_i$ of $P$ in $G_i$ is an ind-definable
Levi subgroup. Since $G = RP$, clearly $G_i = R_i P_i$. Suppose 
$Q_i$ is a perfect subgroup contained in $P_i$ and such that $G_i = R_i Q_i$.
Define $Q$ to be the pre-image of $Q_1 \times \cdots \times Q_k$ in $G$. Then $[Q, Q]$ is a perfect subgroup of $H$ mapping surjectively onto $Q_1 \times \cdots \times Q_k$ and therefore onto $H/Z(G)$. Hence $P = [Q, Q]$ 
and $Q_i = P_i$ for each $i = 1, \dots, k$. 

By the field case, $P_i$ is a conjugate of $S_i$ for each $i = 1, \dots, k$. Say $P_i = S_i^{g_i}$. For any $g \in G$ in the pre-image of $(g_1, \dots, g_k) \in G/Z(G)$, $S^g$ is a perfect
subgroup of $H$ such that $H = Z(G)S^g$. Hence $S^g = P$, as we wanted.

If the structure expands a group, the same argument used for the field case shows that $S$ is the subgroup generated by the image of a definable section $s \colon G/R \to S$  of the canonical homomorphism $S \to G/R$, because $S$ is again a central extension of $G/R$.

 \medskip
 Suppose $Z(G)$ is infinite. Set $\bar{G} = G/Z(G)^0$. The case where $\bar G$ is semisimple has already been considered, so assume $\bar{G}$ is not semisimple and let $\bar R$ be its infinite solvable radical. By induction hypothesis, $\bar{G}$ contains
some ind-definable Levi subgroup $\bar S$. Set $K$ to be the pre-image of $\bar{S}$ in $G$ and $S = [K, K]$ its commutator subgroup. We claim that $S$ and its conjugates are the ind-definable Levi subgroups of $G$.

By induction hypothesis, $\bar{S}$ is a smallest perfect subgroup of $\bar G$ such that $\bar{G} = \bar{R} \bar{S}$. As $K$ is a central extension of the perfect $\bar S$, it has already been showed that $S$ is the \emph{only} perfect subgroup of $K$ such that $K = Z(G)^0 \cdot S$. Since $G = RK$ and $Z(G)^0 \subset R$, $S$ is a smallest perfect subgroup of $G$ such that $G = RS$.

To see that $S$ is unique up to conjugation, assume $P$ is another ind-definable Levi subgroup of $G$ and define $H = Z(G)^0 \cdot P$. Once again, $P = [H, H]$ is the only perfect subgroup of $H$ that is a complement of $Z(G)^0$ in $H$. Therefore, the image of $P$ in $\bar G$ is an ind-definable Levi subgroup, and a conjugate of $\bar S$ by induction hypothesis. Hence $H = K^g$ and $P = [K^g, K^g] = [K, K]^g = S^g$, as we wanted.

Suppose the structure expands a group. If $S = \bigcup X_i$, by saturation there is a definable $X_k$ such that the restriction of the canonical homomorphism $G \to \bar G$ to $X_k$ is surjective. Let $s_1 \colon \bar G \to X_k$ be a definable section.  

By induction hypothesis, there is a definable section $\bar s \colon \bar G/\bar R \to \bar G$
of the canonical homomorphism $\bar G \to \bar G/\bar R$ such that $\bar S$ is the subgroup generated by its image $\bar s(\bar G/\bar R) = \bar X$. 

Note that $\bar G/\bar R$ coincides with $G/R$, so we can assume $G/R$ is the domain of $\bar s$.

Define $s \colon G/R \to X_k$ to be
the composition of $\bar s$ and $s_1$. Clearly, $s$ is a definable section of the canonical homomorphism $G \to G/R$. Set $X = s(G/R)$, $H = \la X \ra$ and $B = Z(G)^0 H$. The image $\bar H$ of $H$ in $\bar G$ is a subgroup containing $\bar X$, so $\bar S \subseteq \bar H$ and $B$ contains $K$. Thus $[B, B] = [H, H]$ contains $[K, K] = S$ and $S \subseteq H$. On the other hand, $X \subset S$ so $H \subseteq S$. Therefore $S = H$, as we wanted.
\end{proof}
 
\begin{rem}
As suggested by a reviewer, the assumption that the structure expands a group may be not necessary for $S$ to be the subgroup generated by a definable section of the canonical homomorphism $G \to G/R$. Although we agree that it is possible the assumption is not necessary, unfortunately we are not able to prove that an appropriate definable section exists in general.
\end{rem}

\begin{rem}
We do not know whether the ind-definable Levi subgroups we find in arbitrary o-minimal structures are always semisimple in the sense of \Cref{dfn:semisimple}. The issue boils down to Question at pg. 87 in \cite{HPP} asking if $Z(G) \cap [G, G]_n$ is finite for each $n$, whenever $G$ is a definably connected central extension of a definable semisimple group.  
\end{rem}


 \section{$0$-groups and $p$-groups}

The goal of this section is to study $p$-groups, for $p=0$ or $p$ prime, and prove \Cref{theo:sylow} and \Cref{cor:p-solvable}.

\begin{fact}\cite[2.6, 2.22]{Strzebonski} \label{fact:hull}
Let $X \subset G$ be a subset (definable or not). There is a smallest definable subgroup $H < G$ containing $X$. We call it \emph{the definable subgroup generated by $X$}, and we write $H = \langle X \df$. If the elements of $X$ commute pairwise, $H$ is abelian.
\end{fact}


\begin{theo}\label{theo:p-groups:list}
Let $G$ be a definable group.  

\begin{enumerate}
\item Suppose $p$ is a prime number. Then $G$ is a $p$-group if and only if $G^0$ is a $0$-group and $G/G^0$ is a finite $p$-group.  In particular, if $G$ is definably connected, $G$ is a $p$-group if and only if $G$ is a $0$-group. \label{pt:p-groups:list:p}

\medskip
\item Suppose $G$ is abelian. Then $G$ is a $0$-group if and only if $G$ is   definably connected and it is the definable subgroup generated by its torsion elements.  
\label{pt:p-groups:list:0}
\end{enumerate}
\end{theo}

\begin{proof}
\begin{enumerate}

\item Suppose $G$ is a definable $p$-group. Then 
\[
E(G/G^0) = |G/G^0| \equiv 0 \mod p.
\]

Suppose, by way of contradiction, $G/G^0$ is not a $p$-group. That is,
$|G/G^0| = np^a$, where $a \geq 1$ and $n > 1$ is not a multiple of $p$. Let $P$ be a $p$-Sylow subgroup of $G/G^0$ and $H$ its pre-image in $G$. Then $E(P) = |P| = p^a$ and $E(G/H) = n$, in contradiction with the fact that $G$ is a $p$-group.  

Next, we want to show that $G^0$ is a $0$-group. 

If $G^0$ is non-trivial, then $E(G^0) = 0$, otherwise $E(G^0) = \pm1$ by \Cref{fact:tricotomy} and $G$
 cannot be a $p$-group. Let $A$ be a $0$-Sylow subgroup of $G^0$. Since $E(G/G^0) = |G/G^0| \neq 0$, $A$ is a 
$0$-Sylow of $G$ too and $E(G/N_G(A)) = 1$ by \Cref{fact:str}(6). But $G$ is a $p$-group, so $G = N_G(A)$ and $A$ is normal in $G$. By \Cref{prop:Anormal}, $G^0$ is solvable and there is a finite subgroup $F$ such that
\[
G/A = (\mtf(G)/\mtf(A)) \rtimes F. 
\] 

\vs \noindent
Let $K$ be the pre-image of $F$ in $G$. Then $E(G/K) = E(\mtf(G)/\mtf(A)) = \pm 1$. Since $G$ is a $p$-group, $\mtf(G) = \mtf(A)$ and $G^0 = \mtf(G) A = A$ 
(\Cref{fact:me-solvable}).  

Conversely, suppose $G^0$ is a $0$-group and $G/G^0$ is a finite $p$-group. Let $H$ be a proper definable subgroup of $G$. We need to show that $E(G/H) \equiv 0 \mod p$.

If $H \cap G^0 = G^0$, then $G^0 \subseteq H$. We have $E(G/G^0) = E(G/H)E(H/G^0) = p^a$, because $G/G^0$ is a finite $p$-group. So $E(G/H) = p^{a - b}$, where $E(H/G^0) = p^b$.

If $H \cap G^0 \subsetneq G^0$, then $E(G^0/(H \cap G^0)) = 0$, because $G^0$ is a $0$-group. Note that $E((G^0H)/H) = 0$, as $G^0/(H \cap G^0)$ is in definable bijection with $(G^0H)/H$. Moreover, 
\[
E(G/H) = E(G/(G^0H)) E((G^0H)/ H) = 0. 
\]

\vs \noindent
In either case, $E(G/H) \equiv 0 \mod p$, and $G$ is a $p$-group.

\medskip
\item If $G$ is a $0$-group, then $G$ is definably connected by \Cref{fact:str}(3). Suppose $K$ is the definable subgroup generated by the torsion elements of $G$. Then $G/K$ is torsion-free and $E(G/K) = \pm 1$ by \Cref{fact:str}. Since $G$ is a $0$-group, it must be $G = K$.

 Conversely, suppose $G$ is definably connected and equal to the definable subgroup generated by its torsion elements. Let $H$ be a proper definable subgroup of $G$. 
If $E(G/H) \neq 0$, then $|E(G/H)| = 1$ by \Cref{fact:tricotomy}. Therefore, 
$G/H$ is torsion-free and $H$ contains the torsion of $G$, in contradiction with the minimality of $G$. So $E(G/H) = 0$ and $G$ is a $0$-group.

\end{enumerate}
\end{proof}

\begin{lem}\label{lem:0-Syl-in-abelian}
Let $S$ be a definable subgroup of a definably connected abelian group $G$.
Then $S$ is the $0$-Sylow of $G$ if and only if $S$ is the definable subgroup generated by the torsion of $G$.
\end{lem}

\begin{proof}
Suppose $S$ is the $0$-Sylow of $G$. Since $G$ is definably connected, the group $G/S$ is definably connected too. Moreover, $|E(G/S)| = 1$ by \Cref{fact:str}(5) and \Cref{fact:tricotomy}. Hence $G/S$ is torsion-free by \Cref{fact:str}(1), and $S$ contains the torsion of $G$. Suppose, by way of contradiction, $H$ is a proper definable subgroup of $S$
containing the torsion of $G$. Then $G/H$ is torsion-free and $E(G/H) = \pm 1$. By \Cref{fact:str}(5), the $0$-Sylow of $H$ is the $0$-Sylow of $G$ too, in contradiction with the fact that $S$ contains $H$ properly.

Conversely, let $S$ be the definable subgroup generated by the torsion of $G$. Since
$G/S$ is torsion-free, $|E(G/S)| = 1$. This means that the $0$-Sylow $S'$ of $S$ is the $0$-Sylow of $G$ too, because 
\[
E(G/S') = E(G/S)E(S/S') \neq 0.
\]
 
On the other hand, $S'$ contains the torsion subgroup of $G$ by the opposite implication above, so $S = S'$, as claimed. 
\end{proof}

\begin{proof}[\textbf{Proof of \Cref{theo:sylow}}]

Let $H < G$ be definable groups and $p$ a prime number.

\begin{enumerate}

\item $(i) \Rightarrow (ii).$ Suppose $H$ is a $p$-Sylow subgroup of $G$. By  \Cref{theo:p-groups:list}\eqref{pt:p-groups:list:p}, 
$H^0$ is a $0$-group. 

If $E(G) \neq 0$, $G^0$ is torsion-free by \Cref{fact:tricotomy} and $G \cong G^0 \rtimes G/G^0$ by \Cref{theo:disconnected}.  It follows that the trivial subgroup is the only $0$-subgroup of $G$ and all definable $p$-subgroups of $G$ are finite. Therefore $H^0 = \{e\} = H \cap G^0$ is the unique $0$-Sylow of $G$. Since $|E(G^0)| = 1$, $|E(G)| = 
|E(G/G^0)|$ and $H$ is a $p$-Sylow of $G/G^0$ by \Cref{fact:str}(5).

Assume $E(G) = 0$. By \Cref{fact:str}(5), $E(G/H) \neq 0 \mod p$. In particular, $E(G/H) \neq 0$. So 
\[
E(G/H^0) = E(G/H)E(H/H^0) \neq 0   
\]

\vs \noindent
and $H^0$ is a $0$-Sylow subgroup of $G$ by \Cref{fact:str}(5) again.

Set $K = N_G(H^0)$ and note that $H \subseteq K$. By \Cref{prop:Anormal}, $K^0$ is solvable and $K^0/H^0$ is torsion-free. Since $(H \cap G^0)/H^0$ is a finite subgroup of $K^0/H^0$, it must be trivial, and $H \cap G^0 = H^0$. It follows that $H/H^0$ is the image of $H$ in $G/G^0$ through the canonical homomorphism $G \to G/G^0$.

Set $G_1 = G/G^0$ and $H_1 = H/H^0$. By \Cref{theo:p-groups:list}\eqref{pt:p-groups:list:p}, $H_1$ is a $p$-subgroup of $G_1$. Let $P_1$ be a $p$-Sylow of $G_1$
containing $H_1$ and $P$ its pre-image in $G$. If $H_1 \neq P_1$, then $E(P_1/H_1) = 0 \mod p$ and
\[
E(G/H) = E(G/P)E(P/H) = E(G/P)E(P_1/H_1) = 0 \mod p
\]

\vs \noindent
in contradiction with $H$ being a $p$-Sylow of $G$. Therefore $H_1 = P_1$ and $H/H^0$
is a $p$-Sylow of $G/G^0$, as we wanted.

\medskip
$(ii) \Rightarrow (i).$ Suppose $H$ is a definable subgroup of $G$ such that $H^0$ is a $0$-Sylow, $H^0 = H \cap G^0$ and $H/H^0$ is a $p$-Sylow of $G/G^0$.

By \Cref{theo:p-groups:list}\eqref{pt:p-groups:list:p}, $H$ is a $p$-group. Moreover,
\[
E(G/H) = E((G/G^0)/(H/H^0)) \neq 0 \mod p
\]

\vs \noindent
because $H/H^0$ is a $p$-Sylow of $G/G^0$. Therefore $H$ is a $p$-Sylow of $G$.

\medskip \item Since a $0$-Sylow is a maximal $0$-subgroup and $0$-groups are abelian, $(i)$ and $(ii)$ are equivalent by \Cref{theo:p-groups:list}\eqref{pt:p-groups:list:0}.
We are left to show that the $0$-Sylow subgroups of $G$ coincide with the $p$-Sylow subgroups of $G^0$ for each $p$.

Suppose $H$ is a $0$-Sylow of $G$. Because $H$ is definably connected, $H \subseteq G^0$ and $H$ is a $p$-Sylow of $G^0$ by $(1)$.

Conversely, suppose $H$ is a $p$-Sylow of $G^0$. By $(1)$, $H = H^0$ and $H$ is a $0$-Sylow of $G^0$. As $E(G/G^0) = |G/G^0| \neq 0$, $H$ is a $0$-Sylow of $G$ by \Cref{fact:str}(5), since
\[
E(G/H) = E(G/G^0) E(G^0/H) \neq 0.
\]  
\end{enumerate}
\end{proof}

\bigskip
 
 \begin{proof}[Proof of \Cref{cor:p-solvable}]
Let $G$ be a $p$-group. By \Cref{theo:p-groups:list}\eqref{pt:p-groups:list:p}, 
$G^0$ is a $0$-group, hence definably generated by its torsion subgroup $T(G^0)$, by \Cref{theo:p-groups:list}\eqref{pt:p-groups:list:0}.
Let $F$ be a finite subgroup from \Cref{theo:disconnected} such that $G = FG^0$. Then 
\[
G = \la T(G^0) \cup F \df.
\]

Since finite $p$-groups are nilpotent and $0$-groups are abelian, it follows from \Cref{theo:p-groups:list} that definable $p$-groups are solvable. 
\end{proof}

\begin{rem}
Unlike finite $p$-groups, definable $p$-groups are not necessarily nilpotent, though. Consider, for instance, the group $G$ from \Cref{ex:inSL2C}. $G$ is a semialgebraic extension of $\Z/2\Z$ by a definable torus, so it is a $2$-group, but it is not nilpotent.
\end{rem}


 \section{A Jordan-Chevalley decomposition}
 
In this section we refine the decomposition established in \Cref{theo:levi} along the lines of \Cref{fact:real-algebraic}. We start by proving a crucial preliminary result following from the characterization in \Cref{theo:sylow}.

 \begin{prop} \label{prop:central}
 Every normal $0$-subgroup of a definably connected group $G$ is central in $G$.
 \end{prop}

\begin{proof}  
Let $K$ be a normal $0$-subgroup of $G$. For each $n \in \N$, the $n$-torsion subgroup $T_n$ of $K$ is definably characteristic in $K$, which is normal in $G$. Therefore, $T_n$ is a normal subgroup of $G$. Fix $x \in T_n$. The definable continuous function $G \to G$ given by $g \mapsto gxg\inv$ has image in $T_n$. However, $G$ is definably connected and $T_n$ is finite, so by taking $g = e$ we deduce that the image must be $\{x\}$, and $x \in Z(G)$. It follows that the torsion subgroup $T = \cup_{n \in \N} T_n$ of $K$ is central in $G$. As $\langle T \df= K$ by \Cref{theo:sylow}, $K$ is central too. 
\end{proof}

\begin{cor}\label{lem:GoverNG}
 Let $G$ be a definably connected group and set $\bar {G} = G/\mtf(G)$. Then $\bar{G}$ is a central extension of a semisimple definable group. 
In particular, $Z(\bar{G})^0$ is the solvable radical of $\bar G$, $[\bar{G}, \bar{G}]$ is the unique ind-definable Levi subgroup, and 
$\bar{G} = Z(\bar{G})^0[\bar{G}, \bar{G}]$.
\end{cor}

\begin{proof}
Let $\pi \colon G \to \bar{G}$ be the canonical projection. If $R$ is the solvable radical of $G$, then $\pi(R)$ is the solvable radical of $\bar{G}$. By \Cref{fact:me-solvable}, $R = \mtf(G) \rtimes T$ and 
$\pi(R) = \pi(T)$ is a normal definable torus of $\bar{G}$. By \Cref{prop:central}, any normal definable torus is central. Therefore the solvable radical of $\bar{G}$ is 
$\pi(R) = \pi(T) = Z(\bar{G})^0$ and $\bar G$ is a central extension of a semisimple group. By the proof of \Cref{theo:levi}, $[\bar G, \bar G]$ is the unique ind-definable Levi subgroup.
\end{proof}

\medskip
 We are now ready to prove \Cref{theo:decomposition}:

\begin{proof}[\textbf{Proof of \Cref{theo:decomposition}}]
\begin{enumerate}

\item For ease of notation, suppose $G$ is definably connected and let $R$ be its solvable radical. We will prove our claims by induction on $\dim G$.

Define inductively a sequence of definable quotients of $G$ as follows. Set $G_0 = G$ and for each $n \in \N$, 
\[
G_{n + 1} = G_n/Z(G_n)^0.
\]

\vspace{.2cm}
There is a smallest $k \in \N$ such that $Z(G_k)$ is finite and therefore $G_{k+i} = G_k$ for each $i \in \N$. Clearly, $\dim G_{i+1} < \dim G_i$ for any $i < k$. Note also that at each step we are quotienting by a definably connected group, so the pre-image in $G_i$
of a definably connected subgroup of $G_{i+1}$ is definably connected.
 
\bigskip
Assume $k = 0$. That is, $Z(G)$ is finite. As shown in the proof of \Cref{theo:levi}, any ind-definable Levi subgroup $S$ is definable and semisimple. 

First, we want to find a $0$-Sylow  $T$ of $R$ centralizing $S$ and such that $G=NTS$, where $N = \mtf(G)$.

Define $H = N_G(S)$, the normalizer of $S$ in $G$. 
Because $E(S)= 0$, $E(H) = E(S) E(H/S) = 0$ as well.   
Let $A$ be a $0$-Sylow of $H$. We claim that $A$ is a $0$-Sylow of $G$ too. By \Cref{fact:str}(5), it suffices to prove that $E(G/A) \neq 0$.

By \Cref{lem:GoverNG}, the image of $S$ in $G/N$ through the canonical homomorphism is the commutator subgroup of $G/N$. In particular, it is a normal subgroup of $G/N$. 
Therefore, for each $g \in G$ there is some $x \in N$ such that $S^g = S^x$. This gives a well-defined definable bijection between definable sets $G/N_G(S)$ and $N/N_N(S)$. Since $N$ is torsion-free and the Euler characteristic is invariant by definable bijections, we have
\[
E(G/H) = E(N/N_N(S)) = \pm 1 
\]

\vs \noindent
by  \Cref{fact:str}(1). Hence
 \[
E(G/A) = E(G/H)E(H/A) \neq 0,
\] 

\vs \noindent
 and $A$ is a $0$-Sylow of $G$, as claimed.  Because $E(G/A) \neq 0$,  
\[
E((AR)/A) = E(R/(R \cap A)) \neq 0,
\]

\vs \noindent 
and  any $0$-Sylow of $R \cap A$ is a $0$-Sylow of $R$ too. Let $T$ be one such $0$-Sylow of $R \cap A$ and set  $K = TS$. We know that $K$ is a (definable) subgroup because $T \subset H$.  

Since $R$ is normal in $G$, it follows that $K \cap R$ is a (solvable) normal subgroup of $K$ and the definable quotient $K/(K \cap R)$ is definably isomorphic to the definable semisimple group $(RK)/R = G/R$, as $K$ contains $S$. Hence $(K \cap R)^0$ is the solvable radical of $K$ and note that $T$ is contained in it. Because $T$ is a $0$-Sylow of $R$, $T$ is a $0$-Sylow of $(K \cap R)^0$ too and $(K \cap R)^0 = \mtf(K) T$ by \Cref{fact:me-solvable}.

On the other hand, $S$ is a normal subgroup of $K$, because $K \subseteq H$. The quotient $K/S = (TS)/S$ is isomorphic to $T/(T \cap S)$. Moreover, 
$S \cap \mtf(K) = \{e\}$ because $S$ is definable semisimple, so $\mtf(K)$ embeds into $T/(T \cap S)$. Since $T \cap S \subseteq R \cap S$
is a finite group, it must be $\mtf(K) \subseteq T$. Therefore $T$ is the solvable radical of $K$.

 By \Cref{prop:central}, $T$ is central in $K$. Hence $T$ is a $0$-Sylow of the solvable radical $R$ centralizing $S$, and $G = NTS$, as wanted.

Conversely, suppose now $T$ is a $0$-Sylow of $R$. We will show there is a definable semisimple Levi subgroup $S$ centralizing $T$.   

Because $Z(G)$ is finite, the quotient $G/Z(G)$ is centerless, so by \cite[Theorem 3.1 \& 3.2]{PPSI} there are definable real closed fields $\Rs_1, \dots, \Rs_k$ such that $G/Z(G)$ is a direct product $G_1 \times \dots \times G_k$ where each $G_i$ is definably isomorphic to a semialgebraic subgroup of $\GL_n(\Rs_i)$, for some $n \in \N$.

Set $T_i$ to be the image of $T$ in $G_i$. Then $T_i$ is a $0$-Sylow $R_i$, the solvable radical of $G_i$, and $G_i = N_G(T_i)R_i$ by \Cref{lem:frattini}. 

Set $H_i = N_G(T_i)^0$. We claim that every definable semisimple Levi subgroup $S_i$ of $H_i$ is a definable semisimple Levi subgroup of $G_i$ too.   

As $G_i$ is definably connected, $G_i = H_iR_i$. The definable group 
\[
\dfrac{G_i}{R_i}\ =\ \dfrac{H_iR_i}{R_i}\ =\ \dfrac{H_i}{H_i \cap R_i} 
\] 

\vs \noindent
is semisimple. Therefore, the definably connected normal (in $H_i$) subgroup $(H_i \cap R_i)^0$ is the solvable radical of $H_i$, $G_i = H_iR_i = S_i (H_i \cap R_i)^0 R_i = S_iR_i$,
and $S_i$ is a definable semisimple Levi subgroup of $G_i$, as claimed.

Since $T_i$ is a normal subgroup of $H_i$, it is central in $H_i$
by \Cref{prop:central}. In particular, $T_i$ centralizes every definable semisimple Levi subgroup $S_i$ of $G_i$ that is contained in $H_i$. 

Set $\bar{S} = S_1 \times \cdots \times S_k$ and $S$ the definably connected component of the pre-image of $\bar{S}$ in $G$. Then $S$ is a definable Levi subgroup of $G$ centralizing $T$, as we wanted. Clearly, $G = RS = NTS$.

\bigskip
Let $k > 0$. Suppose $G_k$ is semisimple. In this case, we will show by induction on $k$ that $G$ contains a unique ind-definable Levi subgroup $S$ and the solvable radical $R$ contains a unique $0$-Sylow $T$.  

If $k = 1$, $G$ is a central extension of a semisimple group, and its commutator subgroup is the unique Levi subgroup by the proof of \Cref{theo:levi}. Since the solvable radical is abelian, it contains a unique $0$-Sylow. Suppose $k > 1$ and let $S_1$ be the unique Levi subgroup of $G_1$ by induction hypothesis. Define $H$ to be the pre-image of $S_1$ in $G$. Any Levi subgroup of $G$ must be contained in $H$, by uniqueness of $S_1$ in $G_1$. We know that $S = [H, H]$ is the unique perfect subgroup of $H$ such that
$H = Z(G)^0 S$. Hence $S$ is the unique Levi subgroup of $G$, as claimed. 
Let $T$ be a $0$-Sylow of the solvable radical $R$ of $G$. Its images $T_1$ in $G_1$ is a $0$-Sylow subgroup of $R_1$, the solvable radical of $G_1$. By induction hypothesis, $T_1$ is unique, hence normal in $G_1$. It follows that its pre-image $K = TZ(G)^0$ is a normal subgroup of $G$. Because $K$ is abelian and definably connected, $T$ is its unique $0$-Sylow and thefore the  unique $0$-Sylow of $R$ too, as claimed.

In particular, $T$ is normal in $G$ and central by \Cref{prop:central}, yelding both implications. In this particular case, the decomposition $G = NTS$ is unique.

\medskip
Assume $G_k$ is not semisimple and set $R_k$ to be its solvable radical.

Fix an ind-definable Levi subgroup $S$ of $G$. One can easily see by induction on $k$ that the image $S_k$ of $S$ in $G_k$ through the canonical homomorpsism is an ind-definable Levi subgroup of $G_k$. By \Cref{theo:levi}, $S_k$ is definable semisimple.

Because $G_k$ has finite center by construction, by the proof of the finite center $k = 0$ case above, there is a $0$-Sylow $T_k$ of $R_k$ that centralizes $S_k$. Note that $T_k$ is a proper subgroup of $R_k$, otherwise $T_k$ would be a normal $0$-subgroup of $G_k$, hence central by \Cref{prop:central}, in contradiction with $Z(G_k)$ being finite.

Let $K$ be the pre-image of $T_kS_k$ in $G$. As noted at the beginning of the proof, $K$ is definably connected because $T_kS_k$ is definably connected. Clearly $S$ is an ind-definable Levi subgroup of $K$ too. As $T_k \neq R_k$ and $S_k$ is definable semisimple, $\dim K < \dim G$. By induction hypothesis, $K = \mtf(K)TS$, where $T$ is a $0$-Sylow  of the solvable radical $\mtf(K)T$ centralizing $S$. Since $G = NK$, $G = NTS$ and we are done.

Conversely, let $T$ be a $0$-Sylow of the solvable radical $R$. The (possibly trivial) image $T_k$ of $T$ in $G_k$ is a $0$-Sylow of the solvable radical $R_k$.

As noted in the proof of the case $k = 0$,
the normalizer of $T_k$ in $G_k$ contains some definable Levi subgroup $S_k$ of $G_k$.
Once again, let $K$ be the pre-image in $G$ of $T_kS_k$. Because $T_k \neq R_k$, $\dim K < \dim G$ and $G = NK$. By induction hypothesis, there is an ind-definable Levi subgroup $S$ of $K$ centralizing $T$. Clearly $S$ is an ind-definable Levi subgroup of $G$ too.

\medskip
\item Fix a decomposition $G^0 = NTS$ as above. We claim that for each $g \in G$, the conjugate $S^g$ is an ind-definable Levi subgroup of $G^0$.

Let $k \in \N$ be the smallest $n$ such that $G^0_n = G^0_{n+1}$ as in $(1)$. We can prove our claim by induction on $k$. If $k=0$,
$S$ is definable semisimple, and so is $S^g$. Thus $R \cap S^g$ is finite and since $\dim S = \dim S^g = \dim (G/R)$, the subgroup $RS^g$ must coincide with $G^0$. 

Suppose $S_1 \subseteq S^g$ is a smallest (perfect) subgroup such that $G^0 = RS_1$.
As noted in the proof of \Cref{theo:levi}, because $S^g$ is a central extension of $G^0/R$, $S_1$ is a normal subgroup of $S^g$ and the group $S^g/S_1$ is abelian, as it is a quotient of $Z(S^g)$. However, $S^g$ is perfect, so $S_1 = S^g$ and $S^g$ is an ind-definable Levi subgroup.

Let $k > 0$. The image of $S^g$ in $G^0_1$ is an ind-definable Levi subgroup by induction hypothesis, hence a conjugate of the image of $S$ in $G^0_1$ by \Cref{theo:levi}. Therefore $S^g$ is a conjugate of $S$ in $G^0$ and an ind-definable Levi subgroup of $G^0$.

It follows that $G = N_G(S)R$. Indeed, because $S^g \subset G^0$ is a ind-definable Levi subgroup, there is $x \in G^0$ such that $S^g = S^x$, by \Cref{theo:levi}. Since $G^0 = SR$, $x=sr$, for some $s \in S$ and $r \in R$, so $S^x = S^r$ and $gr\inv \in N_G(S)$. As $g = (gr\inv) r$, it follows that $G = N_G(S)R$, as claimed. 

Next, we want to show that $N_G(S)$ is definable. First, we can see that the normalizer of $S$ in $G^0$ is definable by induction on $k$. If $k= 0$, $S$ is definable, and so is its normalizer. Let $k > 0$. If $G^0_k$ is semisimple, $S$ is normal in $G^0$.
Assume $G^0_k$ is not semisimple. Let $S_1$ be the ind-definable Levi subgroup of $G^0_1$ corresponding to $S$ (that is, $S_1$ is the image of $S$ in $G^0_1$), $B_1$
the normalizer of $S_1$ in $G_1^0$, and $B$ the pre-image of $B_1$ in $G^0$. 
By the proof of \Cref{theo:levi}, recall that $S$ is the commutator subgroup $[K, K]$ of the pre-image $K$ of $S_1$ in $G^0$ and $S$ is the unique perfect subgroup $P$ of $K$ such that $K = Z(G^0)^0P$. 

We claim that $B$ is the normalizer of $S$ in $G^0$. If $x \in B$, the image of the conjugate $S^x$ in $G_1$ is $S_1$. Therefore $S^x$ is in $H$, and by uniqueness of $S$ it must be $S^x = S$. Hence $B$ is contained in the normalizer of $S$ in $G^0$. The converse is obvious. By induction hypothesis $B_1$ is definable, and so is $B$.

Let $F_1$ be a finite subgroup of $G$ from \Cref{theo:disconnected}, so that for each $g \in G$ there is a definable map $G \to F_1 \times G^0$ such that $g \mapsto (x, g_0)$
with $g = xg_0$. Suppose $f \colon F_1 \to G^0$  is a map such that $S^x = S^{f(x)}$ for each $x \in F_1$. Note that we can take $f$ to be definable because the normalizer of $S$ in $G^0$ is definable. Then for $g = xg_0$, $S^g = S^{xg_0} = S^{f(x)g_0}$ and the map $G \to G^0$ given by $g \mapsto f(x)g_0$ is definable.  
Because $N_{G^0}(S)$ is definable, $N_G(S)$ is definable as well. Note that $T \subseteq N_G(S)$.

 Let $A$ be a $0$-Sylow of $N_G(S)$ containing $T$, and let 
$K$ be the normalizer of $A$ in $N_G(S)$. By \Cref{theo:disconnected}, $N_G(S) = F N_G(S)^0$ for some finite subgroup $F$ contained in $K$ and 
\[
G = N_G(S)R = F N_G(S)^0 R = FG^0
\]

with $F \subset N_G(T) \cap N_G(S)$, as we wanted. In particular, $FT$ and $FS$ are subgroups of $G$, and so are $FN$, $NT$, $NS$, since $N$ is normal in $G$.
Moreover, $NS$ is a normal subgroup of $G$. Indeed, for any $g \in G$, 
\[
(NS)^g = NS^g = NS^{g_0} 
\]

for some $g_0 \in G^0$, as noticed before. Write $g_0 = sta$ with $s \in S$, $t \in T$
and $a \in N$. Then $S^{g_0} = S^{sta} = S^a$ and $NS^{g_0} = NS^a = NS$.

It follows that $FNS$ is a subgroup of $G$ and so are $FNT$ ($NT = R$ is normal in $G$),
$FTS$ (because $N_G(T) \cap N_G(S) \subseteq N_G(TS)$) and $NTS = G^0$.
\end{enumerate}
\end{proof}

Recall that a subgroup $T$ of a definable group $G$ is an \emph{abstract torus} if there is a definable group $G_1$ and a definable homomorphism $f \colon G \to G_1$ whose restriction to $T$ is an abstract isomorphism with a definable torus of $G_1$. A \emph{definable torus} is an abelian definably connected definably compact group.

\begin{fact} \cite[Lem 2.14]{me-nilpotent}\label{fact:abstract-outside-tf}
The intersection of an abstract torus and a torsion-free definable subgroup is trivial in any definable group.
\end{fact}

\begin{proof}
[Proof of \Cref{theo:Nsplitting}] 
\begin{enumerate}[(i)]
\item Assume $G$ is definably connected. By \Cref{theo:levi}, $G$ admits an ind-definable Levi decomposition $G = RS$, where $S$ is ind-definable semisimple. Suppose $S$ is definable.

 Set $N = \mtf(G)$. Let $T_1$ be a $0$-Sylow subgroup of $R$ centralizing $S$ from \Cref{theo:decomposition}, so that $G = NT_1S$. Note that $T_1 = \mtf(T_1) \times T$, where $T$ is an abstract torus, because $T_1/\mtf(T_1)$ is definably compact. Set $H = TS$. We claim that $H$ is a semidirect cofactor of $N$ in $G$. That is, we have to show that $N \cap H = \{e\}$. 

If not, let $x \in N \cap H$ be a non-trivial element. Since $N$ is torsion-free, $x^n$ is non-trivial for any $n \in \N$. On the other hand, $x = ts$, with $t \in T$ and $s \in S$, so
$t\inv x = s \in R \cap S$, since $R = N \rtimes T$ (\Cref{fact:abstract-outside-tf}). Because $S$ is definable semisimple, the solvable normal subgroup $R \cap S$ is central in $S$ and must be finite. Let $n \in \N$ be its cardinality. Hence $(t \inv x)^n = s^n$ is the trivial element. Moreover, because $N$ is normal in $R$,
 $(t \inv x)^n = (t^n)\inv x'$ for some $x' \in N$. Thus $t^n$ is trivial, because $T \cap N$ is trivial. Recalling that $T$ centralizes $S$, we can see that $x^n = (ts)^n = t^n s^n$ is trivial, contradiction. 

Therefore, the sequence splits and $H$ is a semidirect cofactor, as claimed. Since 
$T$ is abelian, it follows that $T \subseteq Z(H)$ and $S = [H, H]$, being $S$ perfect. Moreover, $T \cap S \subseteq R \cap S$ is finite, as previously noted.

Suppose $H$ (or any other semidirect cofactor of $N$) is definable. The subgroup 
$\mtf(G) \mtf(H)$ is the pre-image in $G$ of the maximal normal torsion-free definable subgroup of $G/\mtf(G)$, and it is a normal torsion-free definable subgroup of $G$ by \Cref{fact:str}(1) and \Cref{fact:products}(a). By maximality of $\mtf(G)$, it follows that 
$\mtf(H)$ is trivial. Hence the solvable radical of $H$ is a definable torus and coincides with $Z(H)^0$. Let $A$ be the pre-image in $H$ of a $0$-Sylow $A_1$ of the definable semisimple group $S_H = H/Z(H)^0$. Since $A_1$ and $Z(H)^0$ are definably compact, $A$ is definably compact too.
Note that
\[
E(H/A) = E(S_H/A) \neq 0
\]

so $A$ is a $0$-Sylow of $H$ by \Cref{fact:str}(5). In fact, $A$ is a $0$-Sylow of $G$ too
because $G = NH$ and $|E(N)| = 1$ so $|E(G/A)| = |E(H/A)| \neq 0$.

Since $A$ is a definably compact $0$-Sylow of $G$, every $0$-subgroup of $G$ is definably compact, as claimed. (Note that $0$-groups are definably connected, so $G$ and $G^0$ are interchangeable in the claim.)
 
Conversely, assume every $0$-subgroup of $G$ is definably compact. In particular, $T_1$
is definably compact, $T_1 = T$ is definable and $H = TS$ is definable.

Suppose now $K$ is another definable semidirect cofactor of $N$. We want to show that $K$ is a conjugate of $H$.

Set $T_k = Z(K)^0$ and $H_1 = C_G(T_K)$, the centralizer of $T_K$ in $G$.
Since $T_K \subset Z(K)$, clearly $H_1$ contains $K$ and  $G = N H_1$. Moreover, because $G$ is definably connected and $N$ is torsion-free,
$H_1$ is definably connected too and it is a central extension of the definable semisimple
$G/(NT_1) = G/R$. By \Cref{theo:levi}, the commutator subgroup $S_1$ of $H_1$ is its unique ind-definable Levi subgroup and the unique perfect subgroup $P$ of $H_1$ such that $H_1 = R_1 P$, where $R_1$ is the solvable radical of $H_1$. Since $G = NH_1$, $S_1$ is an ind-definable Levi subgroup of $G$ and must be a conjugate of $S$ by \Cref{theo:levi}. 

On the other hand, $[K, K]$ too is a perfect subgroup of $H_1$ such that $H_1 = R_1[K, K]$. Therefore, it must be $[K, K] = S_1$.  

As both $T$ and $T_K$ are $0$-Sylow subgroups of the solvable radical $R$, they are conjugate to each other, say $T_K = T^g$. Hence $H_1 = C_G(T)^g$, $S_1 = S^g$
and $K = H^g$, as claimed.

\medskip
\item Let $H = TS$ be a semidirect cofactor of $N$ in $G^0$ from $(i)$. By \Cref{theo:decomposition}(2), there is a finite $F < N_G(T) \cap N_G(S)$ such that $G = FG^0$. Define $P = FH$. We claim that $N$ and $P$ have trivial intersection.

Assume $F$ has cardinality $n \in \N$ and let $a \in N \cap P$. Write $a = xh$ with $x \in F$ and $h \in H$. Then $a^n = x^n h' = h'$ for some $h' \in H$, because $H \subset G^0$ is normal in $P$.
However, $N \cap H = \{e\}$, so $a^n = e$. Because $N$ is torsion-free, it must be $a = e$, and $P$ is a semidirect cofactor of $N$ in $G$.

Clearly, $H$ is definable if and only if $P$ is definable. 

Conjugacy follows from (i) and the fact that $F \subset N_G(H)$.
\end{enumerate}
\end{proof}

\begin{proof}[\textbf{Proof of \Cref{cor:cor-linear}}]
If $G$ is a linear group, $G^0$ has definable semisimple Levi subgroups by Theorem 4.5 in \cite{PPSIII}, all $0$-subgroups of $G$ are definably compact by Lemma 2.7 in \cite{me2}, and \Cref{theo:Nsplitting} applies.

Note that for any definable image $G' = \mtf(G') \rtimes P'$ of $G$ in $\GL_n(\Rs)$, the subgroup $P'$ is semialgebraic because $T$ and $S$ are by \cite[Theorem 4.3 \& Theorem 4.6]{PPSIII}.
\end{proof}

\medskip
A natural question is whether it is possible to refine the decompositions in \Cref{theo:decomposition} and \Cref{cor:cor-linear}, when $\mtf(G)$ is not nilpotent, and find a nilpotent torsion-free semidirect factor as in \Cref{fact:real-algebraic}.

\medskip
By Fact 3.5 in \cite{BJO}, in a definable group $G$ the subgroup generated by all normal nilpotent subgroups is normal, definable and nilpotent. So it is the maximal normal nilpotent subgroup and it is definable. We call it the \emph{definable nilradical} of $G$. The definable nilradical of $\mtf{(G)}$  does not always have a semidirect cofactor in $\mtf{(G)}$ (see for instance \S 5, Ex. 6, p.126 in \cite{Bou}), therefore it is not a suitable candidate.  

\medskip
Another definable nilpotent (characteristic in $G$) subgroup of $\mtf{(G)}$ is its commutator subgroup and, in the linear case, the unipotent radical. In the example below the three groups are semialgebraic and all different:

\begin{ex} \label{ex-N}
Let \[N = \left \{ \begin{pmatrix}
1 & x & y &  &   \\
  & 1 &  z &  &   \\
   & &  1 &  &   \\
  &   & & a &   \\
  &   &  & & b \\
\end{pmatrix} \in \GL_5(\R):  a, b > 0 \right\}.\]

$N$ is a nilpotent torsion-free semialgebraic group and, denoted by $U$ its unipotent radical, then $[N, N] \subsetneq U \subsetneq N$. In fact, $[N, N] \cong (\R, +)$ corresponds to the subgroup where $x = z = 0$ and $a = b = 1$, and $U \cong \UT_3(\R)$ is the subgroup where $a = b = 1$.
 \end{ex}

The following is a linear definably connected group $G$ such that the quotient $G/U$ by the unipotent radical $U$ is \emph{not} a central extension of a semisimple group, because the unique complement of $U$ in $\mtf{(G)}$ does not centralize any Levi subgroup of $G$. 

\begin{ex} \label{ex-linear}
Let $N$ be as in the example above and consider the action of $\SL_2(\R)$ on $N$ given by:
\[
\begin{pmatrix}
1 & x & y &  &   \\
  & 1 &  z &  &   \\
   & &  1 &  &   \\
  &   & & a &   \\
  &   &  & & b \\
\end{pmatrix} \mapsto 
\begin{pmatrix}
1 & x & y &  &   \\
  & 1 &  z &  &   \\
   & &  1 &  &   \\
  &   & & e^{\alpha}  &   \\
  &   &  & & e^{\beta}  \\
\end{pmatrix} 
\]

\vs
where, for $A \in \SL_2(\R), \begin{pmatrix}
\alpha \\
\beta 
\end{pmatrix} = A \begin{pmatrix}
\ln a \\
\ln b
\end{pmatrix}$

\vs
In other words, this is the action of matrix multiplication of $\SL_2(\R)$ on $(\R, +)^2$, transferred on $(\R^{>0}, \cdot)^2$ through the exponential map. The resulting linear group $G = N \rtimes \SL_2(\R)$ is definable in the o-minimal structure $\R_{\exp} = (\R, <, + , \cdot, x \mapsto e^x)$. Its unipotent radical is the subgroup
 \[
U =  \left \{ \begin{pmatrix}
1 & x & y &  &   \\
  & 1 &  z &  &   \\
   & &  1 &  &   \\
  &   & & 1 &   \\
  &   &  & & 1 \\
\end{pmatrix} \in \GL_5(\R) \right\}
\]

\vs \noindent
and the quotient $G/U \cong \R^2 \rtimes \SL_2(\R)$ is perfect with  $Z(G/U) = Z(\SL_2(\R)) = \{\pm I\}$.
\end{ex}

\noindent
Since $[N, N] \subseteq U$, $G/[N, N]$ is not a central extension of a semisimple group either. \\

Example \ref{ex-linear} rules out other natural nilpotent candidates in $\mtf(G)$: (1) the commutator subgroup of the solvable radical $R$ (as here $R = N$) and (2) $[\mtf(G), G]$, here coinciding with $Z(N)$ (That is, the subgroup of $N$ where $x = z = 0$) and does not have a semidirect cofactor in $N$, otherwise $N$ would be abelian.

Therefore \Cref{cor:cor-linear} and \Cref{theo:decomposition} seem to be the best possible analogue to the Jordan-Chevalley decomposition in the o-minimal setting, for the linear and general case, respectively.

 \bigskip
{\bf Acknowledgments.} Thanks to Yves de Cornulier for  
the reference in \cite{Bou} of a Lie algebra whose nilradical does not split, 
and for explaining other similar supersolvable non-splitting examples. 
Thanks to several reviewers for their valuable comments and suggestions on previous and current versions of the paper.


\end{document}